\documentclass{amsart}
\usepackage{amsmath}
\usepackage{amsfonts}
\usepackage{amssymb}
\usepackage[utf8]{inputenc}
\usepackage{amssymb}
\usepackage{mathrsfs}
\usepackage{amsthm}
\usepackage{xcolor}
\usepackage{float}
\usepackage{tikz-cd}

\usepackage{enumerate}

\usepackage{graphicx}

\newtheorem{thrm}{Theorem}[section]

\newtheorem{lemma}[thrm]{Lemma}
\newtheorem{prop}[thrm]{Proposition}

\newtheorem{question}[thrm]{Question}
\newtheorem*{claim}{Claim}

\theoremstyle{definition}
\newtheorem{defin}[thrm]{Definition}
\newtheorem{rem}[thrm]{Remark}
\newtheorem{exa}[thrm]{Example}

\DeclareMathOperator{\supp}{supp}

\newcommand{\NN}{\mathbb{N}}

\begin{document}

\title[On the class of NY compact spaces]{On the class of NY compact spaces of finitely supported elements and related classes}
\author{Antonio Avil\'es}
\address{Universidad de Murcia, Departamento de Matem\'{a}ticas, Campus de Espinardo 30100 Murcia, Spain.} \email{avileslo@um.es}
\author{Miko\l aj Krupski}
\address{Universidad de Murcia, Departamento de Matem\'{a}ticas, Campus de Espinardo 30100 Murcia, Spain\\ and \\ Institute of Mathematics\\ University of Warsaw\\ ul. Banacha 2\\
02--097 Warszawa, Poland }
\email{mkrupski@mimuw.edu.pl}

\begin{abstract}
We prove that a compact space $K$ embeds into a $\sigma$-product of compact metrizable spaces ($\sigma$-product of intervals) if and only if $K$ is (strongly countable-dimensional) hereditarily metalindel\"of and every subspace of $K$ has a nonempty relative open second countable subset. This provides novel characterizations of $\omega$-Corson and $NY$ compact spaces. We give an example of a uniform Eberlein compact space that does not embed into a product of compact metric spaces in such a way that the $\sigma$-product is dense in the image. In particular, this answers a question of Kubi\'s and Leiderman. We also show that for a compact space $K$ the property of being $NY$ compact is determined by the topological structure of the space $C_p(K)$ of continuous real-valued functions of $K$ equipped with the pointwise convergence topology. This refines a recent result of Zakrzewski.
\end{abstract}

\subjclass[2020]{Primary: 46A50, 54D30, 54G12}

\keywords{Eberlein compact, Corson compact, Valdivia compact, scattered space, $\sigma$-product}

\maketitle

\section{Introduction}
A compact space $K$ is \textit{Eberlein compact} if $K$ is homeomorphic to a weakly compact subset of a Banach space. This is equivalent to saying that, for some $\Gamma$, the space $K$ is homeomorphic to a compact subset of
$\{x\in [0,1]^\Gamma :\forall \varepsilon>0\;|\{\gamma\in \Gamma:x(\gamma)>\varepsilon\}|<\omega\}$. Given a family $\{X_\gamma:\gamma\in \Gamma\}$ of topological spaces $X_\gamma$ and a point $a=(a_\gamma)_{\gamma\in \Gamma}\in X=\prod_{\gamma\in \Gamma} X_\gamma$, we define the $\sigma$-product in $X$ based at $a$ as
$$\sigma(X,a)=\{x\in \prod_{\gamma\in \Gamma} X_\gamma:\{\gamma\in \Gamma:x_\gamma\neq a_\gamma\}<\omega\}.$$

In this paper we are concerned with the following two, fairly natural subclasses of Eberlein compact spaces.

\begin{defin}
A compact space $K$ is \textit{$\omega$-Corson} if, for some $\Gamma$, $K$ embeds into $\sigma([0,1]^\Gamma)=\{x\in [0,1]^\Gamma:|\{\gamma\in \Gamma:x_\gamma\neq 0\}|<\omega\}$.
\end{defin}

\begin{defin}\label{def2}
A compact space $K$ is \textit{$NY$ compact} if $K$ embeds into $\sigma(\prod_{\gamma\in \Gamma} X_\gamma,a)$, for some family $\{X_\gamma:\gamma\in \Gamma\}$ of compact metrizable spaces and some $a\in \prod_{\gamma\in \Gamma} X_\gamma$.
\end{defin}

Both $\omega$-Corson and $NY$ compacta were recently studied by several authors (see \cite{BKT} \cite{MPZ} and \cite{Z}).
The name $NY$ compacta was introduced in \cite{MPZ} to acknowledge pioneering work by Nakhmanson and Yakovlev on the subject (see \cite{NY}). It easily follows from the Baire category theorem that the Hilbert cube $[0,1]^\omega$ is not $\omega$-Corson. In fact, a compact metric space $K$ is $\omega$-Corson if and only if $K$ is strongly countable-dimensional, i.e., $K$ can be written as a countable union of closed finite-dimensional subspaces (see \cite{MPZ}). However, it is readily seen that every compact metric space is $NY$ compact. It is also easy to show that every $NY$ compact space is Eberlein compact.

The first internal characterizations the class of $NY$ compacta were already established in the seminal paper \cite{NY} by Nakhmanson and Yakovlev. A new description was recently found by Marciszewski, Plebanek and Zakrzewski in \cite{MPZ}.

In the present paper we further contribute to this line of research. Inspired by the celebrated result of Alster \cite{A} characterizing scattered Corson compacta as strong Eberlein ones, we prove that $K$ is $NY$ compact if and only if $K$ is  Corson compact and has the following property, which we refer to as \textit{M-scatteredness}: every nonempty subset of $K$ contains a nonempty relatively open set of countable weight. We also give two other conditions that are equivalent to being $NY$ compact. This is stated as Theorem \ref{characterization_NY_2} below, which is the main result of the paper. A similar theorem is proved for $\omega$-Corson compacta.

In section 4 we study $\omega$-Valdivia and $NY$-Valdivia compact spaces -- classes of compacta whose definitions stem naturally from combining the classical notion of Valdivia compact spaces (see \cite{Ka1}) with the notions of $\omega$-Corson and $NY$ compacta, respectively (see Section 2 for precise definitions).
Both $\omega$-Valdivia and $NY$-Valdivia compacta are easily seen to belong to the class of semi-Eberlein compact spaces introduced and studied in \cite{KL}. Following \cite{KL}, we say that a compact space $K$ is \textit{semi-Eberlein} if, for some $\Gamma$, $K$ embeds into $[0,1]^\Gamma$ in such a way that the set  $\{x\in [0,1]^\Gamma :\forall \varepsilon>0\;|\{\gamma\in \Gamma:x(\gamma)>\varepsilon\}|<\omega\}$ is dense in the image.

The following diagram explains relations between classes of compacta relevant to the subject of the present paper. Perhaps the least obvious implication below is the fact that metrizable compacta are $\omega$-Valdivia; this assertion is due to Marciszewski Plebanek and Zakrzewski (see \cite[Proposition 6.5]{MPZ}).

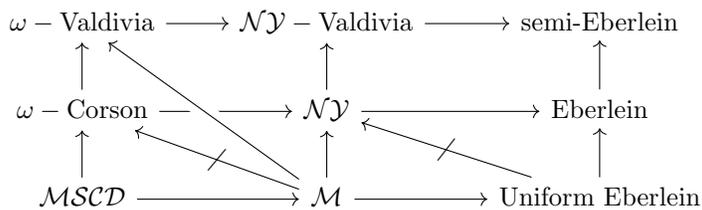
\begin{figure}[H]
\caption{Relations between certain classes of compacta. The symbols $\mathcal{MSCD}$, $\mathcal{M}$ and $\mathcal{NY}$ stand for the class of strongly countable-dimensional metrizable, metrizable and $NY$ compacta, respectively.}
\begin{tikzcd}
\omega-\mbox{Valdivia} \arrow[r]
&  \mathcal{NY}-\mbox{Valdivia} \arrow[r] & \mbox{semi-Eberlein} \\
\omega-\mbox{Corson} \arrow[r] \arrow[u]  &
\mathcal{NY} \arrow[r] \arrow[u]                          & \mbox{Eberlein} \arrow[u]                  \\
\mathcal{MSCD} \arrow[r] \arrow[u]             & \mathcal{M} \arrow[u] \arrow[r] \arrow[lu, "/" marking] \arrow[luu, crossing over]  & \mbox{Uniform Eberlein} \arrow[u] \arrow[lu, "/" marking]
\end{tikzcd}
\end{figure}

As we have already mentioned the Hilbert cube is not $\omega$-Corson, so compact metrizable spaces need not be $\omega$-Corson. Observe that uniform Eberlein compacta need not be $NY$ compact. A suitable example is the countable product $A(\omega_1)^\omega$ of the one-point compactification of a discrete set of size $\omega_1$. This space is uniform Eberlein yet it is not $NY$ compact (see, e.g., Theorem \ref{characterization_NY_1} below).
In Section 4, we will give examples showing that
\begin{enumerate}[(a)]
 \item uniform Eberlein compact spaces need not be $NY$-Valdivia and
 \item $NY$ compact spaces need not be $\omega$-Valdivia.
\end{enumerate}
In particular, this answers Question 6.4 in \cite{KL} and a recent question of H\'{a}jek and Russo \cite[Question 7.6 (2)]{HR}.
Our examples follow from Theorems \ref{theorem-duplicates} and \ref{theorem-duplicates-omega-Corson} below, where we identify all
compacta whose Alexandroff duplicate is $NY$-Valdivia (resp., $\omega$-Valdivia). A similar result concerning Alexandroff duplicates is also proved for the class of semi-Eberlein spaces (see Theorem \ref{semi-Eberlein duplicates} below). This gives a wealth of examples of Corson compact spaces that are not semi-Eberlein (cf. \cite[Example 5.5]{KL} and Remark \ref{remark_plenty of Corson non semi-Eb} below).
So the above diagram may be completed to the following effect:

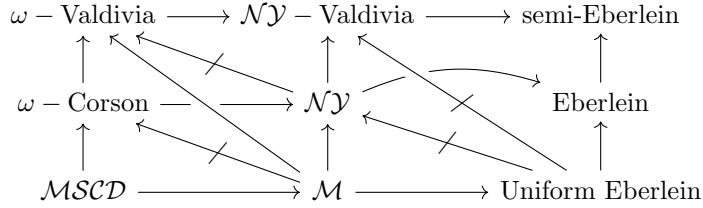
\begin{figure}[H]
\begin{tikzcd}
\omega-\mbox{Valdivia} \arrow[r]
&  \mathcal{NY}-\mbox{Valdivia} \arrow[r] & \mbox{semi-Eberlein} \\
\omega-\mbox{Corson} \arrow[r] \arrow[u]  &
\mathcal{NY} \arrow[r, bend left=20] \arrow[u] \arrow[ul, "/" marking]                         & \mbox{Eberlein} \arrow[u]                   \\
\mathcal{MSCD} \arrow[r] \arrow[u]             & \mathcal{M} \arrow[u] \arrow[r] \arrow[lu, "/" marking] \arrow[luu, crossing over]  & \mbox{Uniform Eberlein} \arrow[u] \arrow[lu, "/" marking] \arrow[luu, "/" marking, crossing over]
\end{tikzcd}
\end{figure}

\bigskip

In section 5 we show that the class of $NY$ compact spaces is invariant under homeomorphisms of $C_p(X)$-spaces, i.e., spaces of continuous functions endowed with the pointwise topology (see Theorem \ref{t-equiv} below). This refines a recent result of Zakrzewski \cite{Z}, who proved analogous result under additional assumption that a homeomorphism in question is linear.

\section{Notation}

All spaces in this paper are assumed to be Tychonoff.
Recall that the weight of a space $X$ is the minimal size of a base for $X$. We say that $X$ is \textit{second-countable} if $X$ has countable weight.

\subsection{$\sigma$-products and $\Sigma$-product}

Let $\{X_\gamma:\gamma \in \Gamma\}$ be a family of spaces and let $a=(a_\gamma)_{\gamma\in \Gamma}$ be a point in the product $\prod_{\gamma\in \Gamma}X_\gamma$. Given
$x=(x_\gamma)_{\gamma\in \Gamma}\in \prod_{\gamma\in \Gamma}X_\gamma$ the support of $x$ with respect to $a$ is the set
$$\supp_a x=\{\gamma\in \Gamma: x_\gamma\neq a_\gamma\}.$$
If $X_\gamma=[0,1]$ for all $\gamma\in \Gamma$ and if $a=(0,0,\ldots)$, then we write $\supp x$ instead of $\supp_a x$.
Similarly if $x=(x_\gamma)_{\gamma\in \Gamma}\in \prod_{\gamma\in \Gamma} X_\gamma$ where  $X_\gamma=\{0,1\}^\omega$ for all $\gamma\in \Gamma$ or  $X_\gamma=[0,1]^\omega$ for all $\gamma\in \Gamma$, then we write
$$\supp x=\{\gamma\in \Gamma:x_\gamma\neq (0,0,\ldots)\}.$$
Given a family of topological spaces $\{X_\gamma : \gamma\in \Gamma\}$ and $a\in \prod_{\gamma\in \Gamma}X_\gamma$ by
$\sigma(\prod_{\gamma\in \Gamma} X_\gamma, a)$ we denote the \textit{$\sigma$-product} in $\prod_{\gamma\in \Gamma} X_\gamma$ centered at $a$ defined as
$$\sigma(\prod_{\gamma\in \Gamma} X_\gamma, a )=\{x\in \prod_{\gamma\in \Gamma} X_\gamma:|\supp_a x|<\omega\}.$$
In the case when every factor $X_\gamma$ is the unit interval $[0,1]$ or every factor $X_\gamma$ is the Cantor set $\{0,1\}^\omega$ or every factor $X_\gamma$ is the Hilbert cube $[0,1]^\omega$, we write
$$\sigma(\prod_{\gamma\in \Gamma} X_\gamma)=\{x\in \prod_{\gamma\in \Gamma} X_\gamma:|\supp x|<\omega\}.$$
If $n\in\omega$, then
$$\sigma_n(\prod_{\gamma\in \Gamma} X_\gamma)=\{x\in \prod_{\gamma\in \Gamma} X_\gamma:|\supp x|\leq n\}$$
is a subset of $\sigma(\prod_{\gamma\in \Gamma} X_\gamma)$ which is closed in the product $\prod_{\gamma\in \Gamma} X_\gamma$. Clearly, $$\bigcup_{n\in \omega} \sigma_n(\prod_{\gamma\in \Gamma} X_\gamma)=\sigma(\prod_{\gamma\in \Gamma} X_\gamma).$$

The \textit{$\Sigma$-products} in $\prod_{\gamma\in \Gamma}X_\gamma$ are defined similarly as collections of all elements of the product $\prod_{\gamma\in \Gamma}X_\gamma$ with countable support, i.e.,
\begin{align*}
 &\Sigma(\prod_{\gamma\in \Gamma} X_\gamma,a)=\{x\in \prod_{\gamma\in \Gamma} X_\gamma:|\supp_a x|\leq\omega\}
\quad \mbox{and} \\
&\Sigma(\prod_{\gamma\in \Gamma} X_\gamma)=\{x\in \prod_{\gamma\in \Gamma} X_\gamma:|\supp x|\leq\omega\}.
\end{align*}

Recall that a compact space which, for some $\Gamma$, is homeomorphic to a subspace of $\Sigma([0,1]^\Gamma)$ is called \textit{Corson compact}.
The following proposition says that in Definition \ref{def2} one can replace a family of arbitrary metrizable compacta by a family consisting of Hilbert cubes.

\begin{prop}\cite[Proposition 3.2]{MPZ}
Given a set $\Gamma$ and $\gamma\in \Gamma$, we let $Q_\gamma=[0,1]^\omega$ to be the Hilbert cube.
 A compact space $K$ is $NY$ compact if and only if $K$ embeds into
 $\sigma(\prod_{\gamma\in \Gamma}Q_\gamma)$, for some $\Gamma$.
\end{prop}

A compact space $K$ is called \textit{Valdivia compact} if, for some $\Gamma$, there is an embedding $h:K\to [0,1]^\Gamma$ such that the set $h(K)\cap \sigma([0,1]^\Gamma)$ is dense in $h(K)$. This motivates the following definitions:

\begin{defin}
A compact space $K$ is \textit{$\omega$-Valdivia} if, for some $\Gamma$, there is an embedding $h:K\to [0,1]^\Gamma$ such that the set $h(K)\cap \sigma([0,1]^\Gamma)$ is dense in $h(K)$.
\end{defin}

\begin{defin}
A compact space $K$ is \textit{$NY$-Valdivia} if, for some $\Gamma$, there is an embedding $h:K\to \prod_{\gamma\in \Gamma}Q_\gamma$ such that the set $h(K)\cap \sigma(\prod_{\gamma\in \Gamma}Q_\gamma)$ is dense in $h(K)$.
\end{defin}

\subsection{Point-finite and point-countable families}
Let $\mathscr{A}$ be a family of subsets of a space $X$. We say that $\mathscr{A}$ is \textit{point-finite} (\textit{point-countable}) in $X$ if for any $x\in X$ the collection $\{A\in \mathscr{A}:x\in A\}$ is finite (countable). Let $\mathscr{U}$ be a cover of a space $X$. We say that a cover $\mathscr{V}$ of $X$ is a \textit{refinement} of $\mathscr{U}$ if for every $V\in \mathscr{V}$ there is $U\in \mathscr{U}$ with $V\subseteq U$.
A space $X$ is \textit{metacompact} (\textit{metalindel\"of}) if every open cover of $X$ has a point-finite (point-countable) open refinement. If every subspace of $X$ is metacompact (metalindel\"of) then we say that $X$ is \textit{hereditarily metacompact} (\textit{hereditarily metalindel\"of}).


\section{Characterizing the class of NY and $\omega$-Corson compacta}

There are various known characterizations of $NY$ compacta that can be found in the literature.
Theorem \ref{characterization_NY_1} below gathers two such descriptions due to Nakhmanson and Yakovlev \cite{NY} (see condition $(ii)$ of Theorem \ref{characterization_NY_1})
and Marciszewski Plebanek and Zakrzewski \cite{MPZ} (see condition $(iii)$ of Theorem \ref{characterization_NY_1}). In order to formulate this result we need to fix some notation first.

Let us recall that that a family $\mathscr{U}$ of subsets of a space $X$ is called $T_0$-separating if for any distinct $x,y\in X$ there is $U\in \mathscr{U}$ with $|U\cap\{x,y\}|=1$. It will be convenient to introduce the following definitions:

\begin{defin}
A family $\mathscr{U}$ of subsets of a space $X$ is called \textit{block-point-finite} if for some set $\Gamma$ one can write $\mathscr{U}=\bigcup\{\mathscr{U}_\gamma:\gamma\in \Gamma\}$ where each family $\mathscr{U}_\gamma$ is countable and the family $\{\bigcup\mathscr{U}_\gamma:\gamma\in \Gamma\}$ is point-finite in $X$.
\end{defin}

\begin{defin}
A space $X$ is called \textit{M-scattered} if every nonempty subspace $A$ of $X$ contains a nonempty relatively open subset of countable weight.
\end{defin}

If a compact space $K$ is M-scattered then for any ordinal number $\alpha$ we define the $\alpha$-th M-derivative $K^{(\alpha)}$ of $K$ as follows (see \cite[section 4]{MPZ}):
\begin{itemize}
 \item $K'=K^{(1)}=K\setminus\bigcup\{U\subseteq K: U \mbox{ is open and second-countable}\}$
 \item $K^{(\alpha+1)}=(K^{(\alpha)})'$
 \item $K^{(\alpha)}=\bigcap_{\beta<\alpha} K^{(\beta)}$ if $\alpha$ is a limit ordinal.
\end{itemize}

We define the M-height $Mht(K)$ of $K$ as
$Mht(K)=\min\{\alpha:X^{(\alpha)}=\emptyset\}$. Note that if $K$ is a nonempty compact space, then we have $Mht(K)=\gamma+1$ for some ordinal number $\gamma$, i.e., the M-height of $K$ is a successor ordinal.

\begin{thrm}\label{characterization_NY_1}(\cite{NY},\cite{MPZ})
 The following conditions are equivalent for any compact space $K$:
 \begin{enumerate}[(i)]
  \item $K$ is $NY$ compact,
  \item there is a $T_0$-separating block-point-finite family consisting of open $F_\sigma$-subsets of $K$
  \item $K$ is hereditarily metacompact and M-scattered.
 \end{enumerate}
\end{thrm}

The purpose of this section is to add yet another three equivalent conditions to the above list. We will prove the following:

\begin{thrm}\label{characterization_NY_2}
The following conditions are equivalent for any compact space $K$:
\begin{enumerate}[(i)]
 \item $K$ is $NY$ compact
 \item $K$ is Corson compact and M-scattered
 \item $K$ is Eberlein compact and M-scattered.
 \item $K$ is hereditarily metalindel\"of and M-scattered.
\end{enumerate}
\end{thrm}

Our proof is based on the ideas of Alster from \cite{A}. In fact, our reasoning is a minor modification of the one given in \cite{A}.

The proof of the next lemma is essentially the same as the proof of \cite[Proposition]{A}). We enclose the argument for the convenience of the reader.

\begin{lemma}\label{lemma_Alster}
 Let $\mathscr{U}$ be a point-countable family of M-scattered clopen subsets of a compact space $K$. There exists a point-finite family $\mathscr{V}$ of clopens in $K$ such that $\mathscr{V}$ refines $\mathscr{U}$ and $\bigcup\mathscr{U}=\bigcup\mathscr{V}$.
\end{lemma}

\begin{proof}
 We proceed by induction on the cardinality of the family $\mathscr{U}$. If $\mathscr{U}=\{U_1,U_2,\ldots\}$ is countable then we put $V_1=U_1$ and $V_n=U_n\setminus \bigcup\{U_k:k<n\}$ for $n>1$. The family $\mathscr{V}=\{V_1,V_2,\ldots\}$ is point-finite being pairwise disjoint.

 Fix an uncountable cardinal $\kappa$. Suppose that the lemma holds if $|\mathscr{U}|<\kappa$ and let
 us assume that $\mathscr{U}=\{U_\alpha:\alpha<\kappa\}$ is of size $\kappa$. For a compact M-scattered set $F$ we define
\begin{equation*}
Z(F)=
  \left\{\begin{aligned}
  &\emptyset &&\mbox{if }F=\emptyset\\
&F^{(\gamma)}  &&\mbox{if }F\neq \emptyset\mbox{ and } Mht(F)=\gamma+1.
\end{aligned}
 \right.
\end{equation*}
Clearly, $Z(F)$ is always a compact metrizable space.
If $\mathscr{W}$ is a family of compact M-scattered sets, then we set
$$Z(\mathscr{W})=\bigcup\left\{Z\left(\bigcap\mathscr{F}\right): \mathscr{F}\in [\mathscr{W}]^{<\omega}\right\}.$$
Define an increasing sequence $\{\mathscr{U}_\beta:\beta<\kappa\}$ of subfamilies of $\mathscr{U}$ as follows:
\begin{equation*}
 \mathscr{U}_\beta=\left\{\begin{aligned}
                & \{U_\alpha:\alpha\leq\beta\}\cup \{U\in \mathscr{U}:U\cap Z(\mathscr{U}_\eta)\neq\emptyset\} && \mbox{ if } \beta=\eta+1\\
                & \bigcup\{\mathscr{U}_\alpha:\alpha<\beta\} && \mbox{ if } \beta \mbox{ is a limit ordinal.}
               \end{aligned}
\right.
\end{equation*}
It is readily seen that
\begin{equation}\label{union is U}
 \mathscr{U}=\bigcup\{\mathscr{U}_\beta:\beta<\kappa\}
\end{equation}

Let us show the following:

\begin{claim}
If $\beta<\kappa$, then $|\mathscr{U}_\beta|\leq |\beta|\cdot\omega<\kappa.$
\end{claim}
\begin{proof}
We will prove the claim inductively with respect to $\beta<\kappa$. If $\beta$ is a limit ordinal then $\mathscr{U}_\beta=\bigcup\{\mathscr{U}_\alpha:\alpha<\beta\}$ and we are done. Suppose that $\beta=\eta+1$ for some ordinal $\eta$ and suppose that the claim holds true for $\eta$, i.e. $|\mathscr{U}_\eta|\leq|\eta|\cdot\omega$. Since
$$\mathscr{U}_\beta= \{U_\alpha:\alpha\leq\beta\}\cup \{U\in \mathscr{U}:U\cap Z(\mathscr{U}_\eta)\neq\emptyset\},$$
it suffices to show that
\begin{equation}\label{eq_claim}
 |\{U\in \mathscr{U}:U\cap Z(\mathscr{U}_\eta)\neq\emptyset\}|\leq |\eta|\cdot\omega.
\end{equation}

Note that if $\mathscr{F}$ is a finite subfamily of $\mathscr{U_\eta}$, then the set $Z\left(\bigcap\mathscr{F}\right)$ is separable, being compact metrizable. For $\mathscr{F}\in [\mathscr{U}_\eta]^{<\omega}$ fix a countable dense subset $D({\mathscr{F}})$ of $Z\left(\bigcap\mathscr{F}\right)$ and put
$$D=\bigcup\left\{D(\mathscr{F}):\mathscr{F}\in [\mathscr{U}_\eta]^{<\omega}\right\}.$$

Clearly, $D$ is dense in $Z(\mathscr{U}_\eta)$ and from $|\mathscr{U_\eta}|\leq |\eta|\cdot\omega$ we get
$|D|\leq|\eta|\cdot\omega$. Hence,
$$\{U\in \mathscr{U}:U\cap Z(\mathscr{U}_\eta)\neq\emptyset\}=\{U\in \mathscr{U}:U\cap D\neq \emptyset\}=\bigcup_{x\in D}\{U\in \mathscr{U}:x\in U\}.$$
Since the family $\mathscr{U}$ is point-countable and $|D|\leq |\eta|\cdot\omega$, the cardinality of the latter family does not exceed $|\eta|\cdot\omega$. This gives \eqref{eq_claim} and finishes the proof of the claim.
\end{proof}

For each $\beta<\kappa$ define

$$\mathscr{W}_\beta=\mathscr{U}_{\beta+1}\setminus \mathscr{U}_\beta.$$
Using \eqref{union is U} and definitions of the families $\mathscr{U}_\beta$ and $\mathscr{W}_\beta$, it is easy to see that
$$\mathscr{U}=\bigcup\{\mathscr{W}_\beta:\beta<\kappa\}.$$ Moreover, by Claim, $|\mathscr{W}_\beta|<\kappa$ for all $\beta<\kappa$.
Applying the inductive assumption to $\mathscr{W}_\beta$ we find a point-finite family $\mathscr{V}_\beta$ consisting of clopen subsets of $K$ such that $\mathscr{V}_\beta$ refines $\mathscr{W}_\beta$ and $\bigcup \mathscr{V}_\beta=\bigcup \mathscr{W}_\beta$. Let
$$\mathscr{V}=\bigcup\{\mathscr{V}_\beta:\beta<\kappa\}.$$
It is readily seen that $\mathscr{V}$ refines $\mathscr{U}$ and $\bigcup\mathscr{U}=\bigcup\mathscr{V}$. It remains to show that $\mathscr{V}$ is point-finite.

Suppose to the contrary that $\mathscr{V}$ is not point-finite and fix $x\in K$ witnessing this. Since each family $\mathscr{V}_\beta$ is point-finite, there is a sequence of ordinals $\beta_1<\beta_2<\ldots$ and distinct sets $V_n\in \mathscr{V}_{\beta_n}$ such that $x\in \bigcap_{n=1}^\infty V_n$. Since $\mathscr{V}_{\beta_n}$ refines $\mathscr{W}_{\beta_n}$, for each $n$ find $W_n\in \mathscr{W}_{\beta_n}$ with $V_n\subseteq W_n$ and put
$$C_n=W_1\cap\ldots \cap W_n.$$
For each $n$ let $Mht(C_n)=\gamma_n+1$ be the M-height of $C_n$. Since the sequence $\{C_n:n=1,2,\ldots\}$ is decreasing we have $\gamma_{n+1}\leq \gamma_n$ and thus there must be $m$ such that $\gamma_{n+1}= \gamma_n$ for all $n\geq m$. From the choice of $m$ we get:
$$W_{m+2}\cap Z(W_1\cap\ldots\cap W_m)\supseteq Z(W_1\cap\ldots\cap W_{m+2}).$$
Since the latter set is nonempty, it follows that
$W_{m+2}\cap Z(W_1\cap\ldots\cap W_m)\neq\emptyset$ so $W_{m+2}\in \mathscr{U}_{(\beta_m+2)}$.
On the other hand
$W_{m+2}\in \mathscr{W}_{\beta_{m+2}}$
so
$$W_{m+2}\notin \mathscr{U}_{\beta_{m+2}}\supseteq \mathscr{U}_{(\beta_{m+1}+1)}\supseteq \mathscr{U}_{(\beta_{m}+2)}$$
which is a contradiction.
\end{proof}

\begin{prop}\label{characterization_zero_dim}
 If $K$ is a zero-dimensional M-scattered Corson compact space, then $K$ has a $T_0$-separating block-point-finite family of clopen subsets of $K$.
\end{prop}
\begin{proof}
 We will prove the proposition inductively with respect to the M-height $Mht(K)$ of $K$. If $Mht(K)=1$, then $K$ is metrizable and we are done. Fix an ordinal $\gamma>1$ and suppose that
 our assertion is true for all $\alpha<\gamma+1$ and that $Mht(K)=\gamma+1$.

 Since $K$ is Corson compact, there is a $T_0$-separating point-countable family $\mathscr{U}$ of open $F_\sigma$-subsets of $K$ (see \cite[U.118]{Tk}). Since $K$ is zero-dimensional we may without loss of generality assume that $\mathscr{U}$ consists of clopen subsets of $K$.

 As in Lemma \ref{lemma_Alster} we put
 $$Z(K)=K^{(\gamma)}.$$
 The set $Z(K)$ is a nonempty compact metrizable subset of $K$.

 Divide $\mathscr{U}$ into two subfamilies $\mathscr{U}=\mathscr{U}_0\cup \mathscr{U}_1$, where
 $$  \mathscr{U}_0=\{U\in \mathscr{U}:U\cap Z(K)=\emptyset\} \quad \mbox{and} \quad
  \mathscr{U}_1=\{U\in \mathscr{U}:U\cap Z(K)\neq\emptyset\}.
 $$
The set $Z(K)$ is compact metrizable so it is separable. Fix a countable dense subset $D$ of $Z(K)$. The family
$$\mathscr{U}_1=\{U\in \mathscr{U}:U\cap D\neq\emptyset\}=\bigcup_{x\in D}\{U\in \mathscr{U}:x\in U\}$$
is countable because $D$ is countable and $\mathscr{U}$ is point-countable.

Applying Lemma \ref{lemma_Alster} to the family $\mathscr{U}_0$ we can find a point-finite family $\mathscr{V}$ of clopen subsets of $K$ so that $\mathscr{V}$ refines $\mathscr{U}_0$ and $\bigcup \mathscr{V}=\bigcup\mathscr{U}_0$. Every $V\in \mathscr{V}$ misses $Z(K)$ so $Mht(V)<Mht(K)=\gamma+1$, for every $V\in \mathscr{V}$.

By the inductive assumption, for every $V\in \mathscr{V}$, there is a family $\mathscr{H}(V)$ of clopen subsets of $V$ which is $T_0$-separating and block-point-finite in $V$. It remains to show the following:
\begin{claim}
 The family $$\mathscr{R}=\mathscr{V}\cup\bigcup\{\mathscr{H}(V):V\in \mathscr{V}\}\cup\mathscr{U}_1$$ is $T_0$-separating and block-point-finite in $K$.
\end{claim}
\begin{proof}
First let us show that the above family $T_0$-separates the points of $K$. To this end pick $x\neq y\in K$ and consider the following three cases:

\medskip

\textit{Case 1:} $x,y\in \bigcup\mathscr{U}_0$. Since $\bigcup \mathscr{V}=\bigcup\mathscr{U}_0$, there is $V\in \mathscr{V}$ such that $x\in V$. If $y\notin V$, then we are done. If $y\in V$, then we can use the family $\mathscr{H}(V)$ to $T_0$-separate the points $x,y$.

\medskip

\textit{Case 2:} $\{x,y\}\cap\bigcup\mathscr{U}_0$ is a one-element set. By symmetry, we can assume that $x\in \bigcup\mathscr{U}_0$ while $y\notin\bigcup\mathscr{U}_0$. We have $\bigcup\mathscr{V}=\bigcup\mathscr{U}_0$ so we can pick $V\in \mathscr{V}$ so that $x\in V$. Since $y\notin\bigcup\mathscr{U}_0$ we have $y\notin V$ and we are done.

\medskip

\textit{Case 3:} $x,y\notin\bigcup\mathscr{U}_0$. The family $\mathscr{U}$ is $T_0$-separating so find $U\in \mathscr{U}$ with $|\{x,y\}\cap U|=1$. Since neither $x$ nor $y$ belongs to $\mathscr{U}_0$ we infer that $U\in \mathscr{U}_1$ and we are done.

So the family $\mathscr{R}$ is $T_0$-separating. Let us show that $\mathscr{R}$ is block-point-finite. For every $V\in \mathscr{V}$, the family $\mathscr{H}(V)$ is block-point-finite in $V$. Thus, for $V\in \mathscr{V}$, we can write:
$$\mathscr{H}(V)=\{\mathscr{H}_t:t\in T(V)\},\quad
\mbox{for some set $T(V)$},$$
where each family $\mathscr{H}_t$ is countable and the family $\{\bigcup\mathscr{H}_t:t\in T(V)\}$ is point-finite in $V$. We may also assume that the sets $T(V)$ are pairwise disjoint for $V\in \mathscr{V}$.

Enumerate $\mathscr{V}=\{V_t:t\in T_0\}$, where $T_0\cap \bigcup\{T(V):V\in \mathscr{V}\}=\emptyset$
and let $a\notin T_0\cup \bigcup\{T(V):V\in \mathscr{V}\}$.
We put
\begin{equation*}
 \mathscr{R}_t=\left\{\begin{aligned}
                & \{V_t\}&& \mbox{ if } t\in T_0\\
                & \mathscr{H}_t && \mbox{ if } t\in \bigcup\{T(V):V\in \mathscr{V}\}\\
                &\mathscr{U}_1  && \mbox{ if } t=a
               \end{aligned}
               \right.
\end{equation*}
and let
$$T=T_0\cup\{a\}\cup\bigcup\{T(V):V\in \mathscr{V}\}.$$
Clearly, for every $t\in T$ the family $\mathscr{R}_t$ is countable. It remains to verify that given $x\in K$, the set
$$T_x=\{t\in T:x\in \bigcup \mathscr{R}_t\}$$
is finite. Fix $x\in K$. If $x\notin \bigcup \mathscr{U}_0$, then $x\notin \bigcup\mathscr{R}_t$ for $t\neq a$. Hence, the set $T_x\subseteq \{a\}$ is at most one-element. If $x\in\bigcup \mathscr{U}_0=\bigcup \mathscr{V}$, then using the fact that $\mathscr{V}$ is point-finite and that the family $\{\bigcup \mathscr{H}_t:t\in T(V)\}$ is point-finite in $V$, for $V\in \mathscr{V}$, we easily check that the set $T_x$ is finite in this case too.
\end{proof}
The proposition is proved.
\end{proof}

The next result is analogous to Theorem 7 in \cite{Y}. Unfortunately no proof of \cite[Theorem 7]{Y} is given in \cite{Y}. It is suggested instead that this theorem can be proved using the same method as \cite[Theorem 8]{Y} (see \cite[p. 275]{Y}). However, van Douwen noted in his review of \cite{Y} (see \cite{vD}) that this method, when applied in the context of \cite[Theorem 7]{Y}, contains a gap. The gap was later filled by Gruenhage and Michael in \cite{GM}. Anyway,
let us give a fairly detailed proof of a more general statement.

\begin{prop}\label{Proposition-Yakovlev}
Let $K$ be a compact space. If $K$ is hereditarily metalindel\"of and M-scattered,then $K$ is Corson compact.
\end{prop}
\begin{proof}
 By \cite[U.118]{Tk}, it is sufficient to show that $K$ has a point-countable $T_0$-separating family of open $F_\sigma$-subsets. To this end, we proceed by induction on $Mht(K)$. If $Mht(K)=1$, then $K$ is metrizable so we are done. Suppose that $Mht(K)=\beta+1$ and that our proposition is already proved for compact spaces of M-height $<\beta+1$.

 The space $X=K\setminus K^{(\beta)}$ is locally compact and metalindel\"of. Note that the family
 $\mathscr{B}$ consisting of all open $F_\sigma$-subsets of $X$ whose closures in $X$ are compact is a base for $X$. According to
 \cite[Corollary 4.1]{GM} we can find a cover $\mathscr{W}\subseteq \mathscr{B}$ of $X$ so that the collection $\{\overline{W}:W\in \mathscr{W}\}$ is point-countable. Here the closures are taken in $X$ but since they are compact, each set $\overline{W}\subseteq X$ is closed in $K$. For each $W\in \mathscr{W}$, we have $\overline{W}\cap K^{(\beta)}=\emptyset$ so $Mht(\overline{W})<\beta+1$. Hence, by the inductive assumption, for every $W\in \mathscr{W}$, we can find a family $\mathscr{V}'(W)$ of open $F_\sigma$-subsets of $\overline{W}$ which $T_0$-separates the points of $\overline{W}$ and
 \begin{equation}
 \mbox{for every } x\in \overline{W} \mbox{ the set } \{V\in \mathscr{V}'(W):x
 \in V\} \mbox{ is countable.}
 \end{equation}
 Let
 $$\mathscr{V}(W)=\{V\cap W:V\in \mathscr{V}'(W)\}.$$
 The family $\mathscr{V}(W)$ consists of open $F_\sigma$-subsets of $K$ because every $W$ is open $F_\sigma$ in $K$ and $V\cap W$ is open $F_\sigma$ in $W$. It is also clear that $\mathscr{V}(W)$ is $T_0$-separating in $W$.

 The set $K^{(\beta)}$ is compact metrizable so we can find a countable family $\mathscr{U}$ of open $F_\sigma$-subsets of $K$ that $T_0$-separates the points of $K^{(\beta)}$ (take a countable base of cozero subsets of $K^{(\beta)}$ and extend them to cozero subsets of $K$). Consider the family
 $$\mathscr{U}\cup \mathscr{W}\cup \bigcup\{\mathscr{V}(W):W\in \mathscr{W}\}.$$
 It is easy to check that this is a point-countable $T_0$-separating family of open $F_\sigma$-subsets of $K$.
\end{proof}

Now we are ready to prove Theorem \ref{characterization_NY_2}
\begin{proof}[Proof of Theorem \ref{characterization_NY_2}]
Every $NY$ compact space is Eberlein compact and by Theorem \ref{characterization_NY_1} every $NY$ compact space is M-scattered, so $(i)\Rightarrow (iii)$. The implication $(iii)\Rightarrow (ii)$ is obvious. According to \cite[Theorem 1]{Y} we have $(ii)\Rightarrow (iv)$ and $(iv)\Rightarrow (ii)$ by Proposition \ref{Proposition-Yakovlev}.

It remains to show $(ii)\Rightarrow (i)$. To this end suppose that $K$ is Corson compact and M-scattered. Since $K$ is Corson compact we may assume that $K\subseteq \Sigma([0,1]^\Gamma)$, for some set $\Gamma$. Let $f:\{0,1\}^\omega\to [0,1]$ be the continuous surjection given by the formula
$$f((x_n)_{n\in \omega})=\sum_{n=0}^\infty \frac{x_n}{2^{n+1}}$$
and let $f_\gamma=f$ for all $\gamma\in \Gamma$.
Consider the product map
$$\varphi=\prod_{\gamma\in \Gamma}f_\gamma:\left( \{0,1\}^\omega\right)^\Gamma\to [0,1]^\Gamma.$$
We claim that $\varphi^{-1}(K)\subseteq \Sigma(\{0,1\}^\omega)^\Gamma$ and thus $\varphi^{-1}(K)$ is Corson compact. True, if $x=(x_\gamma)_{\gamma\in \Gamma}\in (\{0,1\}^\omega)^\Gamma$ and if $\{0,1\}^\omega \ni x_\gamma\neq (0,0,\ldots)$ for some $\gamma\in \Gamma$, then $$\pi_\gamma(\varphi(x))=f_\gamma(x_\gamma)=f(x_\gamma)\neq 0,$$ where $\pi_\gamma:[0,1]^\Gamma\to [0,1]$ is the projection onto the coordinate $\gamma$. So
\begin{equation}\label{supports_inclusion}
 \supp x\subseteq \supp\varphi(x),\quad \mbox{for all } x\in (\{0,1\}^\omega)^\Gamma.
\end{equation}
If $x\in \varphi^{-1}(K)$, then $\supp \varphi(x)$ is countable and thus $\supp x$ is countable as well by \eqref{supports_inclusion}.

Let us show that $\varphi^{-1}(K)$ is M-scattered. To this end fix an arbitrary nonempty subset $A$ of $\varphi^{-1}(K)$. Since $K$ is M-scattered, we can find an open subset $U$ of $K$ so that the set $$U\cap\varphi(A)\subseteq K\subseteq \Sigma([0,1]^\Gamma)$$ is nonempty and second-countable. It follows that the set
$$\Delta=\bigcup\{\supp z:z\in U\cap\varphi(A)\}$$
is countable. Pick $x\in \varphi^{-1}(U)\cap A$. From \eqref{supports_inclusion} we infer that $\supp x \subseteq \Delta$ and hence $\varphi^{-1}(U)\cap A$ may be treated as a subspace of $(\{0,1\}^\omega)^\Delta$. Since $\Delta$ is countable, we conclude that $\varphi^{-1}(U)\cap A$ is second-countable and thus $\varphi^{-1}(K)$ is M-scattered.

Applying Proposition \ref{characterization_zero_dim} and Theorem \ref{characterization_NY_1} to the space $\varphi^{-1}(K)$ we infer that $\varphi^{-1}(K)$ is $NY$ compact. Now, $K$ is $NY$ compact being a continuous image of $\varphi^{-1}(K)$ (see \cite[Corollary 5.3]{MPZ}). This finishes the proof of $(ii)\Rightarrow (i)$.
\end{proof}

Using Theorem \ref{characterization_NY_2} and some results from \cite{MPZ}, we can easily deduce the following theorem that provides a new characterization of the class of $\omega$-Corson compacta.

\begin{thrm}\label{characterization_omega_Corson}
 The following conditions are equivalent for any compact space $K$:
 \begin{enumerate}[(i)]
  \item $K$ is $\omega$-Corson.
  \item $K$ is Corson compact and every nonempty subspace $A\subseteq K$ contains a nonempty relatively open finite-dimensional subspace of countable weight.
  \item $K$ is Eberlein compact and every nonempty subspace $A\subseteq K$ contains a nonempty relatively open finite-dimensional subspace of countable weight.
  \item $K$ is hereditarily metalindel\"of and every nonempty subspace $A\subseteq K$ contains a nonempty relatively open finite-dimensional subspace of countable weight.
 \end{enumerate}
\end{thrm}
\begin{proof}
 Apply Theorem \ref{characterization_NY_1}, Theorem \ref{characterization_NY_2} and \cite[Theorem 4.6]{MPZ}.
\end{proof}

\section{Duplicates}
Given a topological space $X$, the \textit{Alexandroff duplicate of $X$} is the space $AD(X)$ whose underlying set is $X\times\{0,1\}$, endowed with the following topology:
Points in $X\times\{1\}$ are isolated and a basic open neighborhood of $(x,0)$ is of the form
$$(U\times\{0,1\})\setminus \{(x,1)\},$$
where $U$ is an open neighborhood of $x$ in $X$.

It is known (see \cite[U.358]{Tk1} and \cite[Proposition 3.8]{MPZ}) that $AD(K)$ is Eberlein compact ($NY$ compact, $\omega$-Corson) if and only if $K$ is Eberlein compact (resp., $NY$ compact, $\omega$-Corson).
The purpose of this section is to give a complete characterization of compact spaces $K$ whose Alexandroff duplicate $AD(K)$ is semi-Eberlein ($NY$-Valdivia, $\omega$-Valdivia). It turns out that for duplicates being semi-Eberlein ($NY$-Valdivia, $\omega$-Valdivia) is the same as being Eberlein (resp., $NY$ compact, $\omega$-Corson). This enables us to give examples mentioned in the Introduction.

Let us fix some notation.
Given a set $\Gamma$ let us denote
$$c_0(\Gamma)=\{x\in [0,1]^\Gamma:\forall\varepsilon>0\;\;|\{\gamma\in \Gamma:x(\gamma)>\varepsilon\}|<\omega\}.$$

Let $d$ be a metric on a space $X$ (not necessarily related to the topology of $X$). We say that $d$ \textit{fragments} $X$ if for every nonempty closed subset $A$ of $X$ and every $\varepsilon>0$, there is an open subset $U$ of $X$ such that $U\cap  A\neq\emptyset$ and $\sup\{d(x,y):x,y\in U\cap A\}\leq\varepsilon$.

By $d_\Gamma$ we denote the \textit{uniform metric} on the product $[0,1]^\Gamma$, i.e.,
$$d_\Gamma(x,y)=\sup\{|x(\gamma)-y(\gamma)|:\gamma\in \Gamma\},\quad\mbox{for }x,y\in [0,1]^\Gamma$$
In the sequel we shall appeal to the following well known result (see \cite{OSV}, \cite{St}).
\begin{thrm}\label{theorem_OSV}
 A compact space $K$ is Eberlein compact if and only if $K$ is Corson compact and, for some set $\Gamma$, the space $K$ embeds into the product $[0,1]^\Gamma$ in such a way that the uniform metric $d_\Gamma$ fragments the copy of $K$ in $[0,1]^\Gamma$.
\end{thrm}

\begin{lemma}\label{semi_Eb duplicate is Corson}
For any compact space $K$, if the Alexandroff duplicate $AD(K)$ is semi-Eberlein, then $K$ must be Corson compact.
\end{lemma}

\begin{proof}
Fix an embedding $h:AD(K)\to [0,1]^\Gamma$ such that $c_0(\Gamma)\cap h(AD(K))$ is dense in
$h(AD(K))$.
Striving for a contradiction, suppose that $K$ is not Corson. Since $K$ can be identified with the subspace $K\times\{0\}$ of $AD(K)$, there is $x\in K$ such that  $h(x,0)$ has uncountable support. By \cite[Proposition 2.7]{Ka} applied to the compact space $AD(K)$, we can find a homeomorphic embedding of $[0,\omega_1]$ into $AD(K)$ that sends $\omega_1$ to $(x,0)$ and ordinals below $\omega_1$ are sent to points which $h$ maps into the $\Sigma$-product $\Sigma([0,1]^{\Gamma})$. Restricting this embedding to the set of limit ordinals and identifying $K$ with $K\times \{0\}$, we may suppose that there is a homeomorphic embedding
$\phi:[0,\omega_1]\to K$ such that:
\begin{align}
&\phi(\omega_1)=x\\
&h(x,0)\notin \Sigma([0,1]^{\Gamma}),\\
&h(\phi(\alpha),0)\in \Sigma([0,1]^{\Gamma}),\quad \mbox{for }\alpha<\omega_1
\end{align}

Since $h(x,0)$ has uncountable support, we can find a positive integer $N$ such that
\begin{equation}\label{contradiction1}
 \{\gamma\in \Gamma:|h(x,0)(\gamma)|>1/N\} \mbox{ is uncountable.}
\end{equation}

Denote
$L=\phi([0,\omega_1])$. For $n\in \omega$ put
$$A_n=\{z\in [0,1]^\Gamma:|\{\gamma\in \Gamma:z(\gamma)|>1/N\}|\leq n\}.$$
Clearly, each $A_n$ is closed in $[0,1]^\Gamma$. Since the set $c_0(\Gamma)\cap h(AD(K))$ is dense in $h(AD(K))$, it contains all of the isolated points in $h(AD(K))$. Therefore, $$K\times\{1\}\subseteq\bigcup_{n\in \omega}h^{-1}(A_n).$$
The set $L\times\{1\}$ is uncountable so we can find $m\in \omega$ with $h^{-1}(A_m)\cap (L\times\{1\})$ being uncountable. Let $B\subseteq L$ be such that $B\times\{1\}=h^{-1}(A_m)\cap(L\times\{1\})$. Since $B\subseteq L$ is uncountable, the point $x=\phi(\omega_1)\in L$ is an accumulation point of $B$. It follows that the point $(x,0)$ is an accumulation point of the set $B\times\{1\}=h^{-1}(A_m)\cap(L\times\{1\})$. Moreover, since $h^{-1}(A_m)$ is closed in $AD(K)$, we get $(x,0)\in h^{-1}(A_m)$, contradicting \eqref{contradiction1}.
\end{proof}

\begin{lemma}\label{semi-Eb duplicate is fragmentable}
Let $K$ be a compact space. If
$h$ is an embedding of the Alexandroff duplicate $AD(K)$ into $[0,1]^\Gamma$ such that $c_0(\Gamma)\cap h(AD(K))$ is dense in $h(AD(K))$, then the metric $d_\Gamma$ fragments the compactum $h(K\times\{0\})$.
\end{lemma}

\begin{proof}
Fix a closed subset $C$ of $K$ and let $\varepsilon>0$. We need to find an open subset $W$ of $K$ such that $W\cap C\neq \emptyset$ and for all $x,y\in W\cap C$ we have
$$d_\Gamma(h(x,0),h(y,0))\leq \varepsilon.$$
We can assume that $C$ has no isolated points, for otherwise we are easily done.
For $n\in \omega$, let
\begin{align*}
 &A_n=\{z\in [0,1]^\Gamma:|\{\gamma\in \Gamma:z(\gamma)>\varepsilon\}|\leq n\} \mbox{ and }\\ &C_n=\{x\in C:h(x,1)\in A_n\}.
\end{align*}
Since each point of the form $(x,1)$ is isolated in $AD(K)$ and $c_0(\Gamma)$ is dense in the image of $AD(K)$ under the map $h$, we infer that $C=\bigcup\{C_n:n\in \omega\}$. The set $C$ is compact, so by the Baire category theorem
we can find $k$ and a nonempty open subset $U$ of $K$ so that $\emptyset\neq U\cap C\subseteq \overline{C_k}$ (the bar is the closure in $C$).
Let us show that
\begin{equation}\label{h(U)}
 h(x,0)\in A_k \quad\mbox{for all }x\in U\cap C
\end{equation}
Pick $x\in U\cap C$. Since $C$ has no isolated points and $U\cap C\subseteq \overline{C_k}$,
the point $(x,0)$ is in the closure of the set $\{(y,1):y\in C_k\}\subseteq h^{-1}(A_k)$. Clearly, the latter set is closed in $AD(K)$ so $(x,0)\in h^{-1}(A_k)$. This gives \eqref{h(U)}. Put
$$m=\max\{|\{\gamma\in \Gamma:h(x,0)(\gamma)>\varepsilon\}|:x\in U\cap C\}.$$
By \eqref{h(U)}, $m$ is well defined and satisfies $m\leq k$.
Pick $a\in U\cap C$ such that $|\{\gamma\in \Gamma:h(a,0)(\gamma)>\varepsilon\}|=m$. Enumerate
$$\{\gamma_1,\ldots , \gamma_m\}=\{\gamma\in \Gamma:h(a,0)(\gamma)>\varepsilon\}.$$
For each $i\leq m$, put $\delta_i=h(a,0)(\gamma_i)-\varepsilon$ and let
$$\delta=\min\{\varepsilon/2, \delta_1,\ldots , \delta_m\}.$$
Consider the following basic open neighborhood $P$ of $h(a,0)$ in $[0,1]^\Gamma$
$$P=\{z\in [0,1]^\Gamma:|z(\gamma_i)-h(a,0)(\gamma_i)|<\delta\mbox{ for all }i\leq m\}.$$
Since $K$ can be identified with $K\times\{0\}$, there is an open subset $V$ of $K$ such that
$$h^{-1}(P)\cap(K\times\{0\})=V\times\{0\}$$
Let us verify that the set $W=U\cap V$ is as required. First note that $a\in W\cap C$ so the latter set is nonempty and relatively open in $C$. It easily follows from the definition of $m$ and the choice of $\delta$ that
\begin{equation}\label{ten sam zbior}
 \{\gamma\in \Gamma:h(x,0)(\gamma)>\varepsilon\}=\{\gamma_1,\ldots, \gamma_m\}, \mbox{ for }x\in W\cap C.
\end{equation}
Fix $x,y\in W\cap C$ and let $\gamma\in \Gamma$ be arbitrary. If $\gamma\notin\{\gamma_1,\ldots ,\gamma_m\}$,
then $0\leq h(x,0)(\gamma)\leq \varepsilon$ and $0\leq h(y,0)(\gamma)\leq \varepsilon$, by \eqref{ten sam zbior}. Thus, $|h(x,0)(\gamma)-h(y,0)(\gamma)|\leq\varepsilon$. If $\gamma\in \{\gamma_1,\ldots ,\gamma_m\}$, then by the definition of $P$ we have
\begin{align*}
|h(x,0)(\gamma)-h(y,0)(\gamma)|&\leq |h(x,0)(\gamma)-h(a,0)(\gamma)|+|h(y,0)(\gamma)-h(a,0)(\gamma)|<\\
&\delta+\delta\leq \varepsilon/2+\varepsilon/2=\varepsilon.
\end{align*}
Hence, $d_\Gamma(h(x,0),h(y,0))\leq \varepsilon$.
\end{proof}

Using Lemmas \ref{semi_Eb duplicate is Corson} and \ref{semi-Eb duplicate is fragmentable} and appealing to Theorem \ref{theorem_OSV} and \cite[U.358]{Tk1} we get the following result.

\begin{thrm}\label{semi-Eberlein duplicates}
 The following conditions are equivalent for any compact space $K$:
 \begin{enumerate}[(i)]
  \item $K$ is Eberlein compact.
  \item The Alexandroff duplicate $AD(K)$ is Eberlein compact.
  \item The Alexandroff duplicate $AD(K)$ is semi-Eberlein compact.
 \end{enumerate}
\end{thrm}

\begin{rem}\label{remark_plenty of Corson non semi-Eb}
The first example of a Corson compact space which is not semi-Eberlein was given by Kubi\'s and Leiderman in \cite{KL}. Theorem \ref{semi-Eberlein duplicates} can be used to produce different examples of that sort. Indeed, it is sufficient to take a Corson compact space $K$ that is not Eberlein compact (see, e.g., \cite{Ne}). Then $AD(K)$ is Corson compact but not semi-Eberlein according to Theorem \ref{semi-Eberlein duplicates}.
\end{rem}

\begin{thrm}\label{theorem-duplicates}
 The following conditions are equivalent for any compact space $K$:
 \begin{enumerate}[(i)]
  \item $K$ is $NY$ compact
  \item The Alexandroff duplicate $AD(K)$ is $NY$ compact
  \item The Alexandroff duplicate $AD(K)$ is $NY$-Valdivia compact.
 \end{enumerate}
\end{thrm}
\begin{proof}
 The equivalence $(i)\Leftrightarrow (ii)$ is known (see \cite[Proposition 3.8]{MPZ}) and the implication $(ii)\Rightarrow (iii)$ is obvious. It remains to show $(iii)\Rightarrow (i)$. To this end, assume that $AD(K)$ is $NY$-Valdivia and let $$h:AD(K)\to \prod_{\gamma\in \Gamma}Q_\gamma$$ be
 an embedding of $AD(K)$ into a product of the Hilbert cubes $Q_\gamma$ such that the set $h(AD(K))\cap \sigma(\prod_{\gamma\in \Gamma} Q_\gamma)$ is dense in $h(AD(K))$.
 Let
 $$\psi:\prod_{\gamma\in \Gamma}Q_\gamma\to [0,1]^{\omega\times \Gamma}$$
be a homeomorphism that identifies $\prod_{\gamma\in \Gamma}Q_\gamma$ with $[0,1]^{\omega\times \Gamma}$. Note that the set $\Sigma([0,1]^{\omega\times \Gamma})\cap \psi(h(AD(K)))$ is dense in $\psi(h(AD(K)))$.

Since $NY$-Valdivia compact space is semi-Eberlein, we infer from Lemma \ref{semi_Eb duplicate is Corson} that $K$ must be Corson compact. Hence,
according to Theorem \ref{characterization_NY_2}, it is sufficient to check that $K$ is M-scattered.
Suppose to the contrary that $K$ is not M-scattered.
Then for some ordinal $\alpha$ and some nonempty compact subset $C$ of $K$ we have $C=K^{(\alpha)}=K^{(\alpha+1)}$, i.e., the M-derivative stabilizes at some nonempty compact set $C\subseteq K$. In particular, $C$ has no isolated points.

For $n\in \omega$, let
$$A_n=\{x\in C:|\supp h(x,1)|\leq n\}.$$
Since each point of the form $(x,1)$ is isolated in $AD(K)$ and $h$ is an $NY$-Valdivia embedding, we have $h(x,1)\in  \sigma(\prod_{\gamma\in \Gamma}Q_\gamma)$ for every $x\in K$. Hence $C=\bigcup_{n\in \omega}A_n$. By the Baire category theorem, for some $n\in \omega$, there is a nonempty open subset $U$ of $C$ with $U\subseteq \overline{A_n}$. Let us show that
\begin{equation}\label{h(U) is in sigma-product}
 h(x,0)\in \sigma_n(\prod_{\gamma\in \Gamma}Q_\gamma)\quad \mbox{for all }x\in \overline{U}
\end{equation}

Pick $x\in \overline{U}$. Since $C$ has no isolated points,
the point $(x,0)$ is in the closure of the set $\{(x,1):x\in A_n\}\subseteq h^{-1}(\sigma_n(\prod_{\gamma\in \Gamma}Q_\gamma))$. Clearly, the latter set is closed in $AD(K)$ so $\supp(h(x,0))\leq n$. This gives \eqref{h(U) is in sigma-product}. Since $K$ is homeomorphic to the subspace $K\times \{0\}$ of $AD(K)$, it follows from \eqref{h(U) is in sigma-product} that $\overline{U}$ is $NY$ compact. In particular, by Theorem \ref{characterization_NY_1}, $\overline{U}\subseteq C$ is M-scattered. So there is an open subset $V$ of $K$ such that $\emptyset \neq V\cap \overline{U}$ is second countable. Now, the set $U\cap V$ is nonempty open second countable subset of $C$. This contradicts the fact that $C=K^{(\alpha)}=K^{(\alpha+1)}$ and proves that $K$ is M-scattered.
\end{proof}

\begin{exa}\label{example 1}
Let $K=AD(A(\omega_1)^\omega)$ be the Alexandroff duplicate of the countable product $A(\omega_1)^\omega$, where $A(\omega_1)$ is the one-point compactification of a discrete set of size $\omega_1$.
Since the Alexandroff duplicate of a uniform Eberlein compactum is uniform Eberlein compact (see, e.g., \cite[Example 4.2]{KM}), the space $K$ is uniform Eberlein compact.
It is easily seen that $A(\omega_1)^\omega$ is not M-scattered so, by Theorem \ref{characterization_NY_1}, it is not $NY$ compact. Now, it follows from Theorem \ref{theorem-duplicates} that $K$ is not $NY$-Valdivia.
\end{exa}

The above example shows that uniform Eberlein compacta are not necessarily $NY$-Valdivia.
It is well known that every uniform Eberlein compact space is a continuous image of a closed subspace of the space $A(\kappa)^\omega$, where $A(\kappa)$ is the one-point compactification of a discrete set of size $\kappa$ (see \cite{BRW}). This motivates the following question:

\begin{question}
 Is every continuous image of $A(\omega_1)^\omega$ $NY$-Valdivia?
\end{question}

Continuous images of products of the form $A(\kappa)^\lambda$ are called \textit{polyadic spaces}. It is known that if $K$ is polyadic, then the character of $K$ (i.e., the minimal size of a local basis in $K$) and the weight of $K$
coincide \cite[Theorem 6]{G}. Therefore, the space $AD(A(\omega_1)^\omega)$ is not polyadic.

\medskip

For $\omega$-Corson compacta we have the following result analogous to Theorem \ref{theorem-duplicates}:

\begin{thrm}\label{theorem-duplicates-omega-Corson}
 The following conditions are equivalent for any compact space $K$:
 \begin{enumerate}[(i)]
  \item $K$ is $\omega$-Corson.
  \item The Alexandroff duplicate $AD(K)$ is $\omega$-Corson
  \item The Alexandroff duplicate $AD(K)$ is $\omega$-Valdivia compact.
 \end{enumerate}
\end{thrm}
\begin{proof}
 The equivalence $(i)\Leftrightarrow (ii)$ is known (see \cite[Proposition 3.8]{MPZ}) and the implication $(ii)\Rightarrow (iii)$ is obvious. We need to show $(iii)\Rightarrow (i)$. To this end, assume that $AD(K)$ is $\omega$-Valdivia and let $$h:AD(K)\to [0,1]^\Gamma$$ be
 an embedding such that $h(AD(K))\cap \sigma([0,1]^\Gamma)$ is dense in $h(AD(K))$. We will check that $K$ satisfies condition (ii) of Theorem \ref{characterization_omega_Corson}.
 Note that by Theorem \ref{theorem-duplicates}, the space $K$ is Corson compact. So it remains to verify that
 every nonempty subset $A$ of $K$ has a relatively open finite-dimensional second-countable subspace.

  To this purpose, fix a nonempty subset $A\subseteq K$. Let $C$ be he closure of $A$ in $K$. According to Theorem \ref{theorem-duplicates} the space $K$ is $NY$ compact, so $C$ contains a nonempty relatively open second-countable subset $U$ (by M-scatteredness of $K$; cf. Theorem \ref{characterization_NY_1}).
 For $n\in \omega$, let
$$A_n=\{x\in U:|\supp h(x,1)|\leq n\}.$$
Since each point of the form $(x,1)$ is isolated in $AD(K)$ and $h$ is an $\omega$-Valdivia embedding, we have $h(x,1)\in  \sigma([0,1]^\Gamma)$ for every $x\in K$. Hence $U=\bigcup_{n\in \omega}A_n$. Since $U$ is a Baire space (being an open subspace of a compact space $C$), for some $n\in \omega$ there is an nonempty open subset $V$ of $U$ with $V\subseteq \overline{A_n}$. As in the proof of Theorem \ref{theorem-duplicates}, we show that
\begin{equation}\label{h(U) is in sigma-product2}
 h(x,0)\in \sigma_n([0,1]^\Gamma)\quad \mbox{for all }x\in \overline{V}.
\end{equation}

It is known that the dimension of the space $\sigma_n([0,1]^\Gamma)$ does not exceed $n$ (see \cite{EP}). Since $K$ is homeomorphic to the subspace $K\times \{0\}$ of $AD(K)$, it follows from \eqref{h(U) is in sigma-product2} that $C$ (and hence $A$) contains a finite-dimensional second-countable relatively open subset.
\end{proof}

\begin{exa}\label{Example 4.6}
Let $Q=[0,1]^\omega$ be the Hilbert cube. It is an easy consequence of the Baire category theorem that $Q$ is not strongly countable-dimensional. So $Q$ is not $\omega$-Corson (cf. \cite[Corollary 5.1]{MPZ}). Of course, $Q$ is $NY$ compact being metrizable. It follows form Theorems \ref{theorem-duplicates-omega-Corson} and \ref{theorem-duplicates}  that the duplicate $AD(Q)$ is not $\omega$-Valdivia but it is $NY$-compact. In particular, this example gives an affirmative answer to a question posed by Kubi\'s and Leiderman (see \cite[Question 6.4]{KL}). Since $AD(Q)$ is also uniform Eberlein compact (cf. \cite[Example 4.2]{KM}), it provides an answer to a recent question asked by H\'{a}jek and Russo (see \cite[Problem 7.6 (2)]{HR}).
\end{exa}

\section{$NY$ compacta are invariant under homeomorphisms of $C_p$-spaces}

Given a space $Z$, by $[Z]^{<\omega}$ we denote the family of all finite subsets of $Z$. Recall that a map $T:X\to [Y]^{<\omega}$ is \textit{upper semicontinuous} if for any open subset $U$ of $Y$, the set $\{x\in X:T(x)\subseteq U\}$ is open in $X$. We do not require here that $T(x)\neq\emptyset$. A map $f:X\to Y$ is \textit{finite-to-one} if $f^{-1}(y)\in[X]^{<\omega}$ for all $y\in Y$. A continuous surjection $f:X\to Y$ is said to be \textit{irreducible} if no proper closed subset of $X$ maps onto $Y$. A compact space $K$ is \textit{$\omega$-monolithic} if $\overline{A}$ is second countable for any countable $A\subseteq K$. Every Corson compact space is $\omega$-monolithic (see \cite[U.120]{Tk}).

The proof of the main result of this section will be based on the following theorem of Okunev.

\begin{thrm}\cite[Theorem 1.1]{O}\label{thm Okunev}
Let $X$ and $Y$ be Tychonoff spaces. Suppose that there is an open continuous surjection of $C_p(X)$ onto $C_p(Y)$. Then, there is a sequence of upper semicontinuous maps $T_n:X^n\to  [Y]^{<\omega},\;n\in \mathbb{N}$, such that $\bigcup\{T_n(X^n):n\in \mathbb{N}\}=Y$.
\end{thrm}

We will also need the following theorem which may be of independent interest.

\begin{thrm}\label{Lemma fin-to-1}
Let $K$ be an $\omega$-monolithic compact space.
If $K$ is a finite-to-one preimage of a metrizable compactum,  then $K$ contains a nonempty open second-countable subset.
\end{thrm}

\begin{proof}
Fix a finite-to-one map $f:K\to L$ of $K$ onto a metrizable compactum $L$. Striving for a contradiction suppose that no nonempty open subset of $K$ is second-countable. Recursively, construct two sequences $K_1,K_2,\ldots$ and $U_0, U_1,\ldots$ of nonempty subsets of $K$, such that for all $n$ the following conditions are satisfied:
\begin{enumerate}
 \item $K_n$ is compact;
 \item $U_n$ is open in $K$;
 \item $U_{n+1}\subseteq \overline{U_{n+1}}\subseteq U_n$;
 \item $\overline{U_n}\cap K_n=\emptyset$ and $K_{n+1}\subseteq \overline{U_n}$;
 \item $f(\overline{U_n})=f(K_{n+1})$ and the map
 $f\upharpoonright K_{n+1}:K_{n+1}\to f(\overline{U_n})$ is irreducible.
\end{enumerate}
Put $U_0=K$ and find $K_1\subseteq K$ such that $f\upharpoonright K_1:K_1\to L$ is irreducible (see \cite[S.366]{Tk1}). Fix $n\geq 0$ and suppose that the sets $U_n$ and $K_{n+1}$ are already defined in such a way that the conditions (1)--(5) are satisfied.
By condition (5), the restriction of $f$ to the set $K_{n+1}$ maps $K_{n+1}$ irreducibly onto a metrizable compactum $f(\overline{U_n})$. It follows that $K_{n+1}$ has a countable $\pi$-base (cf. \cite[S.228, Fact 1]{Tk1}) and, in particular, $K_{n+1}$ is separable. Since $K$ is $\omega$-monolithic, $K_{n+1}$ must be second countable. By our assumption, $U_n\setminus K_{n+1}\neq\emptyset$ for otherwise $K$ would contain a nonempty second-countable subset $U_n$. Let $U_{n+1}$ be a nonempty open subset of $K$ satisfying
$$U_{n+1}\subseteq \overline{U_{n+1}}\subseteq U_n\setminus K_{n+1}.$$
Consider the map $g=f\upharpoonright \overline{U_{n+1}}:\overline{U_{n+1}}\to f(\overline{U_{n+1}})$ and find a compact subset $K_{n+2}\subseteq \overline{U_{n+1}}$ such that the restriction of $g$ to $K_{n+2}$ is irreducible (see \cite[S.366]{Tk1}). This finishes the recursive construction.

According to (3), there is $x\in \bigcap_{n=0}^\infty \overline{U_n}$. Put $y=f(x)$. Since $x\in \overline{U_n}$ for all $n$, we have $f^{-1}(y)\cap K_n\neq\emptyset$ for all $n$, by (5). Since the sequence $(K_n)_{n=1}^\infty$ is pairwise disjoint (see (4)) we infer that the fiber $f^{-1}(y)$ is infinite which is impossible because $f$ is finite-to-one.
\end{proof}

\begin{prop}\label{proposition finite-to-1}
Let $K$ be an Eberlein compact space. If $K$ is a finite-to-one preimage of an M-scattered compactum, then $K$ is M-scattered.
\end{prop}
\begin{proof}
Let $f:K\to L$ be a continuous finite-to-one surjection of $K$ onto an M-scattered compact space $L$. Fix a nonempty subset $A$ of $K$. It is sufficient to show that the closure $\overline{A}$ of $A$ in $K$ contains a nonempty relatively open second-countable subset (because then $A$ contains such a subset as well). Of course $\overline{A}$ is an Eberlein compact space.

Let $h:\overline{A}\to f(\overline{A})$ be the restriction of $f$ to $\overline{A}$. The map $h$ is a finite-to-one surjection.
Since $L$ is M-scattered, there is a nonempty open second-countable subset $U$ of $f(\overline{A})$. Let $V$ be a nonempty open subset of $f(\overline{A})$ satisfying $\overline{V}\subseteq U$. The set $\overline{V}$ is metrizable and compact.

Let $g$ be the restriction of $h$ to $\overline{h^{-1}(V)}$. Since $\overline{h^{-1}(V)}\subseteq h^{-1}(\overline{V})$, the range of $g$ is a subset of $\overline{V}$, thus it is compact metrizable. We can therefore apply Theorem \ref{Lemma fin-to-1} to the map $g$. Consequently, we get an open subset $W$ of $\overline{A}$ such that $W\cap \overline{h^{-1}(V)}$ is nonempty and second-countable, whence $W\cap h^{-1}(V)$ is nonempty, second-countable and open in $\overline{A}$.
\end{proof}

The following lemma noted by Zakrzewski \cite{Z} is an easy consequence of the Baire category theorem:

\begin{lemma}\label{lemma Zakrzewski}
If a compact space $K$ is a countable union of M-scattered compacta, then $K$ is M-scattered.
\end{lemma}

\begin{thrm}\label{t-equiv}
 If the spaces $C_p(K)$ and $C_p(L)$ are homeomorphic, then $K$ is $NY$ compact if and only if so is $L$.
\end{thrm}
\begin{proof}
By symmetry it is sufficient to show that if $K$ is $NY$ compact, then $L$ is $NY$ compact too. So let us assume that $K$ is $NY$ compact. In particular, $K$ is Eberlein compact and since the class of Eberlein compacta is invariant under homeomorphisms of $C_p$-spaces (see \cite[Corollary IV.1.8]{Ar}), the space $L$ is Eberlein compact as well. According to Theorem \ref{characterization_NY_2} it suffices to show that $L$ is M-scattered. Let $(T_n)_{n\in \mathbb{N}}$ be a sequence of upper semicontinuous maps $T_n:K^n\to [L]^{<\omega}$ provided by Theorem \ref{thm Okunev}. For each $n\in \mathbb{N}$, let $G_n=\{(\mathbf{x},y)\in K^n\times L:y\in T_n(\mathbf{x})\}$ be the graph of $T_n$. Since $T_n$ is upper semicontinuous, it easily follows that $G_n$ is closed in $K^n\times L$. Hence, $G_n$ is an Eberlein compact space for all $n\in \mathbb{N}$ (see \cite[Theorem III.3.4]{Ar} and \cite[Proposition III.3.5]{Ar}). Since the values of each $T_n$ are finite subsets of $L$, the set $G_n$ is a finite-to-one preimage of a closed subspace of $K^n$. By Proposition \ref{proposition finite-to-1}, the space $G_n$ is M-scattered and thus $NY$ compact (cf. Theorem \ref{characterization_NY_2}). Let $\pi_n:K^n\times L\to L$ be the projection. For every $n\in \NN$, the set $\pi_n(G_n)=T_n(K^n)$ is $NY$ compact, in particular M-scattered, because the class of $NY$ compacta is invariant under taking continuous images (see \cite[Corollary 5.3]{MPZ}). According to Theorem \ref{thm Okunev} we have
$$\bigcup\{\pi_n(G_n):n\in \mathbb{N}\}=\bigcup\{T_n(K^n):n\in \mathbb{N}\}=L$$
so, by Lemma \ref{lemma Zakrzewski}, the space $L$ is M-scattered.
\end{proof}

We do not know if analogous result is true for the class of $\omega$-Corson compacta. In the light of Theorem \ref{t-equiv} and \cite[Corollary 5.1]{MPZ} this reduces to the following:

\begin{question}
 Suppose that $K$ and $L$ are $NY$ compact spaces and let $K$ be strongly countable-dimensional. Suppose further that the spaces $C_p(K)$ and $C_p(L)$ are homeomorphic. Must $L$ be strongly countable-dimensional?
\end{question}

\section*{Acknowledgements}
We wish to thank the referee for valuable comments, and in particular for pointing out a more general statement for Theorem \ref{Lemma fin-to-1}.

The authors were partially supported by
Fundaci\'{o}n S\'{e}neca - ACyT Regi\'{o}n de Murcia project 21955/PI/22, Agencia Estatal de Investigación (Government of Spain) and Project PID2021-122126NB-C32 funded by
MICIU/AEI /10.13039/\\
501100011033/ and FEDER A way of making Europe (A. Avil\'es and M. Krupski); European Union - NextGenerationEU funds through Mar\'{i}a Zambrano fellowship and
the NCN (National Science Centre, Poland) research Grant no.\\
2020/37/B/ST1/02613 (M. Krupski)

\end{document}